\documentclass{amsart}
\newcommand{\Ol}{\mathrm{Ol}}
\newcommand{\Z}{\mathbb{Z}}
\newcommand{\C}{\mathbb{C}}
\newtheorem{Theo}{Theorem}
\newtheorem{Lem}{Lemma}
\begin{document}
\title{An Improvement on Olson's Constant for $\Z_p\oplus \Z_p$}
\author[G. Bhowmik]{Gautami Bhowmik}
\author[J.-C. Schlage-Puchta]{Jan-Christoph Schlage-Puchta}
\begin{abstract} We prove that for a prime number $p$ greater than 6000, 
the Olson's constant for  the group $\Z_p\oplus \Z_p$ 
is given by 
$\Ol(\Z_p\oplus \Z_p)=p-1+\Ol(\Z_p)$.
 
\end{abstract}
\maketitle
\section{Introduction }
Let $G$ be a finite additive abelian group of order $n$. A subset $A$ of $G$ 
is 
said to be a zero-sum set if the sum of all its elements is zero. and $\Ol(G)$,
the Olson's constant of $G$, is defined to be the smallest integer $k$ such that
every set of $k$ elements of $G$ contains a zero-sum subset. 

The exact value of this constant is only known for a few cases. As far as bounds are concerned,
Szemer\'edi\cite{Szem} proved the Erd\H os-Heilbronn conjecture that 
$\Ol(G)\le c\sqrt n$,
$c$ being an absolute constant. For cyclic groups, the conjectural value of $c$ (due to Erd\H os and Graham) 
$\sqrt 2$, was recently attained by Nguyen, Szemer\'edi and Vu\cite{NSV}. The conjecture
was  verified  by Gao, Ruzsa and Thangadurai\cite{Ruzsa} for $\Z_p\oplus \Z_p$ 
for all 
$p>4.67\times 10^{34}$. They in fact proved that $\Ol(\Z_p^2)=p-1+\Ol(\Z_p)$ for
such a $p$. Our aim is to improve the bound for $p$, and we prove that

\begin{Theo}
\label{thm:Olson}
Let $p>6000 $ be a prime number. Then $\Ol(\Z_p^2)=p-1+\Ol(\Z_p)$.
\end{Theo}
Our proof falls into two parts, the first one being combinatorial and
dealing with the case where the elements of $A$ are not
well-distributed over $\Z_p^2$, the second one being analytical, using
exponential sums.
Unfortunately, our bound is still too large
to allow for explicit computations. 
Though our method could be used to lower the bound for $p$ further, we
would not be able to go below $p<200$.

Many similar zero-sum problems have been studied, one among them being 
the Davenport's constant where the objects are multi-sets rather than sets. 

\section{Proof}
For a set $A$, we use $\Sigma(A)$ for the set of all its subset sums
while $\Sigma_k(A)$ denotes  the set of all  sums of those subsets which have
$k$ elements.
We will use the fact that $\Ol(\Z_p)\geq\lfloor\sqrt{2p}\rfloor$\cite{Ham}. 

The following is due to Dias da Silva and Hamidoune\cite{DdaSH}.
\begin{Lem}
\label{Lem:ksubsets}
Let $A\subseteq\Z_p$ be a set, $k$ an integer in the range $1\leq
k\leq |A|$. Then we have
\[
\left|\Sigma_k(A)\right|\geq\min(p, k(|A|-k)+1).
\]
In particular, if $|A|\geq \ell:=\lfloor\sqrt{4p-7}\rfloor+1$, and
$k=\lfloor \ell/2\rfloor$, then $\Sigma_k(A)=\Z_p$.
\end{Lem}

\begin{Lem}
\label{Lem:CD}
For $A, B\subseteq\Z_p$ we have $|A+B|\geq\min(p, |A|+|B|-1)$.
\end{Lem}

The following result was proven by Olson \cite[Theorem~2]{Olson}.
\begin{Lem}
\label{Lem:Olson}
Let $A\subseteq\Z_p$ be a set with all elements distinct and $|A|=s$. Suppose
that for all $a\in A$, $-a\not\in A$; in particular, $0\not\in A$. Then we have
\[
|\Sigma(A)|\geq\min(\frac{p+3}{2}, \frac{s(s+1)}{2}+\delta),
\]
where
\[
\delta=\begin{cases} 1, & s\equiv 0\pmod{2}\\
0, & s\equiv 1\pmod{2}
\end{cases}.
\]
\end{Lem}

In the sequel let $A\subseteq\Z_p^2$ be a zero-sum free set of size
$p-1+\Ol(\Z_p)$. Our aim is to show the following:
\begin{Theo}
\label{thm:Olson+}
Let $A\subseteq\Z_p^2$ be a zero-sum free set of size
$p-1+\Ol(\Z_p)$. Then there exists a subgroup $U\cong\Z_p$, such that
$|A\cap U|=\Ol(\Z_p)$, and all other elements of $A$ are contained in
one coset of $U$.
\end{Theo}

Clearly, Theorem~\ref{thm:Olson+} implies Theorem~\ref{thm:Olson}. For
an affine subspace $x+U$ and a set $B$ define $N(x, U, B)=|B\cap(x+U)|$. Set
$M=M(A)=\max\limits_{x, U} N(x, U, A)$.

\begin{Lem}
\label{Lem:2p/5}
Suppose that $M\geq 2p/5$. Then Theorem~\ref{thm:Olson+} holds true
for $A$. 
\end{Lem}
\begin{proof}
This follows immediately from \cite[Lemma~3.4]{Ruzsa} and
\cite[Lemma~3.5]{Ruzsa}.
\end{proof}

The following is the main technical result of the combinatorial part.
\begin{Lem}
\label{Lem:Comb}
Let $U$ be a non-trivial subgroup, and $B\subseteq A$ a set. Let
$\pi:\Z_p^2\rightarrow \Z_p$ be a projection with kernel $U$, and let
$x_1, \ldots, x_p$ be representatives of $\Z_p^2/U$. Suppose
that the multi-set $\pi(B)$ represents each element of $\Z_p$ as a
(possibly empty) subset sum, and that $\sum_i \lfloor N(x_i, U, A\setminus
B)/2\rfloor \lceil N(x_i, U, A\setminus B)/2\rceil \geq p+1$. Then
$A$ contains a zero-sum.
\end{Lem}
\begin{proof}
Among each set $(A\setminus B)\cap(x_i+U)$ we choose all subsets of size
$\lfloor N(x_i, U, A\setminus B)$ and add them up. By
Lemma~\ref{Lem:ksubsets} we obtain in this way at least $\lfloor
N(x_i, U, A\setminus B)/2\rfloor \lceil N(x_i, U, A\setminus
B)/2\rceil$ elements in $\Z_p^2$, each of which has the same image
under $\pi$. By Lemma~\ref{Lem:CD} and the assumption we find that
there exists some $x$, such that every element of $x+U$ is a subset
sum of $A\setminus B$. On the other hand, there is a subset of $B$
with sum contained in $(-x)+U$, hence, we can combine a subset sum of
$B$ with a subset sum of $A\setminus B$ to become a zero-sum. 
\end{proof}

We shall repeatedly apply this Lemma to reduce the size of the numbers
$N(x, U, A)$.

\begin{Lem}
\label{Lem:MMigros}
Suppose that $p\geq 29$ and $2p/5\geq M\geq \lfloor\sqrt{4p-7}\rfloor +1$. Then
$A$ contains a zero-sum. 
\end{Lem}
\begin{proof}
Let $U$ be a subgroup such that there exist
some $x$ with $N(x, U, A)=M$, and let $\pi$ be a projection with
kernel $U$. We choose $\lfloor\sqrt{4p-7}\rfloor +1$ elements in one
coset, and let $B$ be the complement of this set. Consider the
multi-set $\overline{B}=\pi(B)$. Then $|\overline{B}|\geq
p-\sqrt{p}$. Moreover, $\overline{B}$ contains no element with
multiplicity $\geq 2p/5$, hence, in $B$ we can find a system of
$p/5-\sqrt{p}$ disjoint subsets containing 3 different elements, that
is, we find $p/5-\sqrt{p}$ subsets containing two different elements,
which are not inverse to each other. Hence, we have
\[
p\geq |\Sigma(B)| \geq \min(p, 3(p/5-\sqrt{p}) + (p-2(p/5-\sqrt{p})) = 
\min(p, 6/5 p - \sqrt{p}) = p,
\]
and we see that we can apply Lemma~\ref{Lem:Comb} to obtain our claim.
\end{proof}

We now combine Lemma~\ref{Lem:Comb} with an estimate for exponential
sums to obtain a criterion for our theorem to hold which is
numerically applicable.

\begin{Lem}
\label{Lem:Numerik}
Let $p>800$ be a prime number. Let $A\subseteq\Z_p^2$ be a subset with
$|A|=p+\mathrm{Ol}(\Z_p)$. For a subgroup $U\cong\Z_p$ fix a
complement $V$, and define $\lambda_j^U=N(j, U, A)$, where $j$ is
viewed as an element of $V$ via the isomorphism $\Z_p\cong V$. 
Suppose that one
of the following two conditions holds true.
\begin{enumerate}
\item There exists a subgroup $U$, such that the following holds
  true. Denote by $J$ the set of indices $j$ such that $\lambda_j$ is
  odd. Suppose there exists a set of integers $I\subseteq\Z_p$, such
  that $\lambda_i\geq 1$ for all $i\in I$, $\Sigma(I\cup J)=\Z_p$, and
  $\sum_i \lfloor\lambda_i^*/2\rfloor\lceil\lambda_i^*/2\rceil\geq
  p-1$, where
\[
\lambda_i^*=\begin{cases} \lambda_i-1, & i\in I,\\
\lambda_i, & \mbox{otherwise.}
\end{cases}
\]
\item For all subgroups $U$ and all isomorphisms $\Z_p\cong V$ we have the bound
\begin{equation*}
\prod_{i=0}^{p-1}|\cos\frac{j\pi}{p}|^{\lambda_j}\leq\frac{1}{p^2}.
\end{equation*}
\end{enumerate}
Then every subset $A\subseteq\Z_p^2$ with $|A|=p+\mathrm{Ol}(\Z_p)$ contains a
zero-sum. 
\end{Lem}
\begin{proof}
Let $A$ be a subset of $\Z_p^2$ with $|A|=p+\Ol(\Z_p)$.

Suppose that there exists a subgroup $U$, such that for the partition
$\lambda_i=N(U, i, A)$ the first condition holds true. Set $x=\sum
i\lfloor\lambda_i^*/2\rfloor$. By assumption we can choose a subset of
$I\cup J$ adding up to $-x$, let $I', J'$ be the intersection of this
set with $I$ and $J$, respectively. Then we choose elements $x_j$ in
$A\cap(j, 0)+(0, 1)\Z_p$ for all $j\in I'\cup J'$, these elements sum
up to an element $s$ with first coordinate $-x$. Hence, if we choose a
set $A_j$ consisting of $\lfloor\lambda_i^*/2\rfloor$ elements in $(j,
0)+(0, 1)\Z_p$, then $\sum_j\sum_{a\in A_j} a + \sum_{j\in I'\cup J'}
x_j$ has first coordinate 0. To prove that $A$ contains a zero-sum, it
suffices to show that by choosing the sets $A_j$ in all possible ways,
all elements in $(0, 1)\Z_p$ can be reached, and from
Lemma~\ref{Lem:ksubsets} and \ref{Lem:CD} we see that this is the case
if $\sum_j \lfloor\lambda_j^*/2\rfloor\lceil\lambda_j^*/2\rceil\geq
p-1$, thus, the first condition is sufficient.

Hence, we may assume that for each subgroup $U$ the partition $N(i, U,
A)$ satisfies the second condition. Write $e(x)=e^{2\pi i x/p}$; we
view this as a function $e:\Z_p\rightarrow\C$. Then
using orthogonality we see that the number of subsets of $A$ adding up
to 0 equals
\[
\frac{1}{p^2}\sum_{\alpha\in\Z_p^2} \prod_{a\in A} 1+e(\langle a,
\alpha\rangle).
\]
Clearly, the summand $\alpha=0$ contributes $\frac{2^{|A|}}{p^2}$. We
have
\begin{multline*}
\prod_{a\in A} |1+e(\langle a, (0, 1)\rangle)| =
\prod_{j\in\Z_p}|1+e(j)|^{N(j, \langle(0, 1)\rangle, A)}\\
 = 2^{|A|}\prod_{j\in\Z_p}|\cos(\pi j/p)|^{N(i, \langle(0, 1)\rangle,
   A)}
 \leq \frac{2^{|A|}}{p^2},
\end{multline*}
where in the last step we used the second condition. Hence, the number
of zero-sums is bounded from below by
\[
\frac{2^{|A|}}{p^2} - \frac{p^2-1}{p^2}\frac{2^{|A|}}{p^2} =
\frac{2^{|A|}}{p^4}\geq 2,
\]
provided that $p\geq 11$, that is, there exists a non-empty subset
with sum 0.
\end{proof}

Note that the two conditions in the lemma work in different
directions: While the first condition says that most of the
$\lambda_j$ are small, the second condition says that most of the
weight of the partition lies on indices $i$ which are close to 0 or
to $p$, from this difference we shall obtain our result.

\begin{Lem}
\label{Lem:ManyDistinct}
Suppose that $p>1024$ and that there exists a subgroup $U$ such that the
image $\pi(A)$ of the 
projection has less than $p/5$ elements. Then $A$ contains a zero-sum.
\end{Lem}
\begin{proof}
We choose a subset $B\subset A$, such that all elements of $\Z_p$ can
be represented as subset sums of $B$, and $|A\setminus B|\geq p$. Set
$f(\ell)=\lfloor\ell/2\rfloor\lceil\ell/2\rceil$. If $\sum_x f(N(x, U,
A\setminus B))\geq p$, we can apply Lemma~\ref{Lem:Comb}. The function
$f$ is convex, hence, we have
\[
\sum_x f(N(x, U, A\setminus B))\geq \frac{f(5)}{5}|A\setminus B| =
\frac{6}{5}|A\setminus B|, 
\]
and it suffices to show that we can choose $B$ sufficiently small. Suppose
first that the projection of $A$ onto $U$ contains at least
$\sqrt{4p-7}$ different elements. Then we take arbitrary different
elements and obtain our claim, provided that
$p/5+\Ol(\Z_p)\geq\sqrt{4p-7}$, which is certainly the case for
$p>100$. If the projection of $A$ onto $U$ contains less elements,
there are $p/2$ elements in $A$ contained in pre-images of $\pi_U$ of
single points, which contain $\sqrt{p}/4$ elements. Let $B$ be the
complement of this set. Again from
convexity we see that $\sum_x f(N(x, U, A\setminus B))\geq p$,
provided that $\sqrt{p}/4\geq 8$, which is the case for $p>1024$. On
the other hand, the remaining points may be partitioned into sets
containing $p/2$ elements altogether, and no $2\sqrt{p}$ have the same
image under $\pi_U$, hence, we see that $\Sigma(B)=\Z_p$ as well.
\end{proof}

\begin{Lem}
Suppose that $p>6000$. Then $A$ contains a zero-sum.
\end{Lem}
\begin{proof}
For every subgroup we can select $p/5$ elements with different value
under $\pi_U$. Since $4p/5+\Ol(\Z_p)>\sqrt{4p-7}$, it suffices to show
that the second condition of Lemma~\ref{Lem:Numerik} is satisfied for
each set consisting of $p/5$ different elements.
We have
\begin{eqnarray*}
\log \prod_{j=-p/10}^{p/10}\cos(j\pi/p) & < &
\int_{-p/10}^{p/10}\log\cos(t\pi/p)\;dt\\ 
 & = & p\int_{-1/10}^{1/10}\log\cos(t\pi)\;dt\\
 & < & -0.00332296p,
\end{eqnarray*}
hence, our claim follows provided that $p^2<1.003328^p$, which is the
case for $p>6000$.
\end{proof}

There are several obvious ways to improve the argument. First, $p/5$
in Lemma~\ref{Lem:ManyDistinct} can be improved, but not beyond
$p/4$. Then, $\frac{1}{p^2}$ in the second condition of
Lemma~\ref{Lem:Numerik} can be improved, since the exponential sum
will have a smaller value most of the time. However it will be difficult to
ensure that for some subgroup there will be no large term, that is, we
do not expect to obtain anything better then $\frac{1}{p}$. Finally,
one could consider the set of all partitions explicitly in the second 
part of Lemma~\ref{Lem:Numerik}, the improvement here is certainly
smaller then the bound obtained by taking $p/4$ elements four times
each. However, none of these improvements is completely
straightforward and even
if we suppose that the technical difficulties could be
overcome, our method cannot reach $p=200$, since the computational
amount would increase dramatically -- in particular for enumerating
all partitions of $p$. Hence we do not attempt to push our method to its
limits. Still we did formulate Lemma~\ref{Lem:Numerik} in a more general
way than we actually needed to help eventual improvements.

\end{document}